\documentclass[12pt,a4paper]{amsart}

\usepackage{latexsym}
\usepackage{amsmath}
\usepackage{amssymb}
\usepackage{amsthm}
\usepackage{amsxtra}
\usepackage{amsfonts}
\usepackage{mathrsfs}
\usepackage{amsbsy}
\usepackage{graphicx}
\usepackage{hyperref}

\usepackage{stackengine}

\newtheorem{teo}{Theorem}[section]
\newtheorem{dfn}[teo]{Definition}
\newtheorem{lem}[teo]{Lemma}

\newtheorem{prop}[teo]{Proposition}
\newtheorem{remark}[teo]{Remark}

\numberwithin{equation}{section}

\renewcommand{\Im}{\operatorname{Im}\,}
\renewcommand{\Re}{\operatorname{Re}\,}

\newcommand{\HH}{\mathcal H}

\newcommand{\DD}{\mathcal D}

\newcommand{\al}{\alpha}
\newcommand{\ga}{\gamma}

\newcommand{\la}{\lambda}

\newcommand{\ome}{\omega}

\newcommand{\ba}{\begin{eqnarray}} \newcommand{\ea}{\end{eqnarray}}
\newcommand{\be}{\begin{equation}} \newcommand{\ee}{\end{equation}}
\newcommand{\bdm}{\begin{displaymath}} \newcommand{\edm}{\end{displaymath}} 
\newcommand{\brr}{\begin{array}}\newcommand{\err}{\end{array}}

\newcommand{\ri}{\right}

\providecommand{\abs}[1]{\lvert #1 \rvert}

\newcommand{\bml}{\begin{gather}} 
\newcommand{\eml}{\end{gather}}

\newcommand{\beq}{\begin{equation}}
\newcommand{\eeq}{\end{equation}}
\newcommand{\RE}{\mathbb{R}}

\def\CO{{\mathbb C}}

\renewcommand{\leq}{\leqslant}
\renewcommand{\geq}{\geqslant}
\renewcommand{\epsilon}{\varepsilon}

\newcommand{\x}{{\bf x}}

\renewcommand{\k}{{\bf k}}
\newcommand{\0}{{\bf 0}}

\title[]{Blow-up and instability of standing waves\\ for the NLS
with a point interaction in dimension two}

\author[]{Domenico Finco}
\address{Facolt\`a di Ingegneria, Universit\`a Telematica
Internazionale Uninettuno,  Corso Vittorio Emanuele II 39, 00186 Roma, Italy}
\email{d.finco@uninettunouniversity.net}

\author[]{Diego Noja}
\address{Dipartimento di Matematica e Applicazioni, Universit\`a
 di Milano Bicocca,  via Roberto Cozzi 55, 20126 Milano, Italy}
\email{diego.noja@unimib.it}

\begin{document}

\begin{abstract} 
In the present note we study the focusing NLS equation in dimension two with a point interaction and in the supercritical regime, showing two results. 
After obtaining the (nonstandard) virial formula, we exhibit a set of initial data that blow-up. Moreover we show that the standing waves $e^{i\omega t} \varphi_\omega$ corresponding to ground states $\varphi_\omega$ of the Action functional are strongly unstable, at least for sufficiently high $\omega$. \end{abstract}

\maketitle

\begin{footnotesize}
 
 \end{footnotesize}




\section{Introduction}
In the present paper we study the blow-up of solutions of a focusing Nonlinear Schr\"odinger equation (NLS) with a power nonlinearity in two dimension and in the $L^2$ supercritical regime, perturbed by a point defect. The point defect is represented as a point interaction, sometimes improperly called delta potential. Namely, we consider the model
\begin{equation}\label{pinco}
\left\{\begin{aligned}
i \dot\psi (t)  &= {\HH_\alpha} \psi (t) - |\psi|^{p-1}\psi \\
 \psi (0)  &=  \psi_0
\end{aligned} \right.\end{equation}
where $\mathcal H_\al$ is defined as a self-adjoint extension of the symmetric operator $-\Delta$ starting from the domain $C^\infty(\RE^2\setminus\{0\})$, and $\alpha $ is a parameter classifying the self-adjoint extension. A typical feature of the point interactions $\mathcal H_\al$ is that its operator domain $\mathcal D_\al$ or its energy domain $\mathcal D^{\frac{1}{2}}_\al$ (see Section 2.1 for details) are larger than the corresponding domains for the Laplacian, respectively the Sobolev spaces $H^2$ and $H^1$. This is the reason why they have to be considered as singular perturbation of the Laplacian operator (see Section 2.1 for details or the treatise \cite{Albeverio}).  
The operator $\mathcal H_\al$, that can be properly defined only in dimension $n\leq 3$, describes a \emph{zero range interaction}, meaning that the interaction is concentrated at a point. In Quantum Mechanics this fact is exploited to describe situations in which the details of the interactions
are irrelevant and the effective behavior of the system is well described by the Hamiltonian $\mathcal H_\al$, where a single physical parameter characterizes the behavior of the system. This occurs for example in system of non-relativistic particles at low temperature, where the thermal wavelength is much larger then the range of the two body interactions, so that the only effective parameter is the scattering length, directly related to $\al$ (see \cite{Albeverio} for extensive treatment and bibliography). In the case of Nonlinear Schr\"odinger equation \ref{pinco} in which a nonlinear continuous medium is considered, for example a Kerr medium in fiber optics or also a Bose-Einstein condensate, both described in suitable approximation by the NLS equation, the singular perturbation of the Laplacian given by $\mathcal H_\al$ is typically interpreted as the presence in the medium of a defect perturbing the wave propagation.
This model has been studied extensively in one dimension, where a wealth of results have been obtained as regards well posedness, blow-up, existence of standing waves and their orbital and asymptotic stability, with several variation on the theme (see \cite{GHW, FOO08, LeCFFKS08, OY16, CM19, GO20, MMS20} and references therein for a sample of the literature). \\ The model in dimension two and three has been tackled only recently. The well posedness of the two dimensional model has been given first in the strong setting, i.e. for solutions in the operator domain   $\mathcal D_\al$ in \cite{CFN21} (where also the three dimensional case is treated). Then the problem as been settled in in the energy space, i.e. for solutions in $\mathcal D^{\frac{1}{2}}_\al$ in \cite{FGI22} (see Section 2.2.1 below for the state-of-the-art of well-posedness results). 
The critical nonlinearity power in dimension two, namely the power nonlinearity above which global well posedness is not anymore granted, as in the unperturbed model is $p=3$ (notice in this respect the rather different behavior of the model studied in \cite{ACCT20}). In this paper we want to give information about the blow-up of solutions for $p>3$. We will firstly show that for definite and large classes of initial data $\psi_0$ one has a finite existence time $T^*(\psi_0)$. Then we will show strong instability  behavior around ground states of the action, i.e. existence of blowing-up states in any neighborhood of such ground states. The starting point is the formula for the second derivative of the variance, or virial identity, obtained in Section 3 (see Lemma \ref{lem2sec3}). Such a formula contains an anomalous term with respect to the standard unperturbed model, which is positive definite and not conserved by the evolution, and that prevents a simple identification of an invariant set of initial data that blow-up. To overcome the issue, we adopt a strategy originally developed in the classical paper \cite{BC81} for the unperturbed model (see for more details Section 8.2 in \cite{caz}). However one has to suitably modify the analysis, exploiting the variational properties of the action functional $S_\omega$ on the Nehari manifold (see Section 2.2.2 for definitions and further details). Existence and properties of the ground state $\varphi_\omega$ of the action have been studied in \cite{ABCT22} and \cite{FGI22}. In particular $\varphi_\omega$ exists for any $\alpha$ and for any $\omega> -E_\al$ where $-E_\al$ is the always existing eigenvalue of $\mathcal H_\al$. Our first main result gives a class of initial data (containing an open set in the phase space) that undergoes blow-up. In the statement below, $E$ is the total energy \eqref{energy}, $Q$ is the functional defined in \eqref{defQ}) and $\Sigma_\alpha$ is the subset of finite energy states $\mathcal D_\al^{\frac{1}{2}}$ with $\|\x\psi\|_{2}<\infty$ (see formula \eqref{Sigmaal}).
\begin{teo}\label{blowup1}
Let $p > 3$ and $\psi_0 \in \Sigma_\al $.
\ Suppose that $S(\psi_0)<S_\omega(\phi_\omega)\ ,$ $E(\psi_0)\geq 0\ ,$ and $Q(\psi_0)<0\ .$ Then $T^*(\psi_0) < +\infty\ .$\\
\end{teo}
\noindent
We notice that analogous results with a similar strategy have been already obtained in different models, including the already mentioned one dimensional delta interaction (see in particular \cite{FO18, O19}); in this case however the delta term is a form perturbation of the Laplacian, and in this sense it can be considered a standard potential, allowing for an easier treatment in comparison to the present model. Notice also that the virial identity (see \eqref{virial2}) needs a somewhat different treatment than the standard formula; in particular we analyze the transformation properties of the mass preserving map $T^\sigma$ 
given by dilatations (see Proposition \ref{Tsigma1}). The second result concerns the strong instability of the standing waves, i.e. the fact that in the vicinity of any standing waves there are solutions that blow-up (see Definition \ref{stronginstdef}). This fact, again following the ideas contained in \cite{BC81} and in the cited papers related to more standard potential perturbation of the Laplacian, is contained in the second main result.
\begin{teo}\label{stronginst1}
Let $p > 3$, $\omega>|E_\alpha|$ and $\varphi_\omega$ a ground state of the action $S_\omega $ with $E(\varphi_\omega)>0\ .$  Then the standing wave $e^{i\omega t}\varphi_\omega $ is strongly unstable.
\end{teo}
It is well known that in the unperturbed NLS equation, the ground states with $p>3$ have positive energy, while in the present case one expects positive energy only for sufficiently big $\omega $ in analogy with the known case of the presence of external potentials (See Remark \ref{omegagrande}). \\
In the last Section we consider a different condition for the strong instability of standing waves, replacing the positive total energy with the more general condition $\frac{d^2}{d\sigma^2}S_\omega(T^\sigma\varphi_\omega)\vert_{\sigma=1} \leq0 $. It is known that this requirement is sufficient to guarantee strong instability in the case of NLS with a delta potential in one dimension and with the generalized Coulomb potential with arbitrary dimension (see \cite{FO18, O19}). We show first that an invariant set of blowing-up initial data exist (see Theorem \ref{blowup2}), and then that the Action ground states with  $\frac{d^2}{d\sigma^2}S_\omega(T^\sigma \varphi_\omega) \vert_{\sigma=1}\leq0 $ belong to the norm closure of this set and so they are strongly unstable (see Theorem \ref{stronginst2}).\\
We end the introduction noticing that the above results and their proofs actually do not depend on the $\alpha$ parameter. For this reason and to ease formulae and reading we will omit the pedice $\alpha$ from Section 3 and Section 4 in which proof of the main results are given.

\section{Preliminaries}
\subsection{Point interaction in 2d} In the following we shall denote with a boldface points in $\RE^2.$ Let us recall, see for example chapter II.4 of \cite{Albeverio}, that for $n=2$ the operator $\HH_\alpha$ has the domain: 
\begin{equation*}
D (\HH_\alpha) = \left\{ \psi \in L^2 (\RE^n)|\;\psi  = \phi^\la   +q \,  G^\lambda ,\, \phi^\la \in H^2(\RE^2), \;\; q=(\Gamma^\lambda_\alpha)^{-1} \,  \phi^\la(\0)\ri\} 
\end{equation*}
with $G^\lambda$ fundamental solution of the laplacian and $\Gamma^{\lambda}_\alpha$ a certain fixed constant. Explicitly (indicating from now on by the symbol $\mathscr F$ the Fourier transform)
\begin{align}\label{Glambda}
G^\lambda&:=(-\Delta+\lambda)^{-1}\delta_\0=\frac{1}{2\pi}\mathscr F^{-1}\left[\frac{1}{|k|^2+\lambda} \right]=\frac{1}{2\pi}K_0(\sqrt\lambda|\x|)\\
\Gamma^{\lambda}_\alpha&:= 
 {\al +\frac{1}{2\pi}\ga +\frac{1}{2\pi}\ln(\sqrt{\la}/2) } \qquad  
\qquad \quad \ \ \alpha \in\RE\cup\{+\infty\}.\nonumber \end{align}
Here $K_0$ is the MacDonald function of order zero and $\gamma$ is the Euler-Mascheroni constant. The constant $\al$ is real (nontrivial interaction) or $+\infty$ ($q=0$, corresponding to the standard Laplacian). It enters in the relation $ \phi^\la(\0)=\Gamma^\lambda_\alpha\ q $, playing the role of a boundary condition at the singularity and more concretely it is related to the s-wave scattering length $a_0$ through the relation  $a_0 = (-2\pi\alpha)^{-1}$.
The number $ \la $ can be any number in $\RE^+\setminus\{-E_\alpha\}$, where $E_\alpha$ is the negative eigenvalue of $\HH_\alpha$, always existing when $\alpha\in \RE$ (see later).\ 
The action of the operator  is given by 
\begin{equation*}
(\HH_\alpha +\la)\psi = (-\Delta +\la )\phi^\la \qquad (\iff \HH_\alpha \psi = -\Delta \phi^\la -\lambda q G^\la) \qquad \forall\psi\in D(\HH_\alpha)\end{equation*}
It is easily seen and well known that while the decomposition in regular part $\phi^\lambda$ and singular part $qG^\lambda$ of any element $\psi\in D(\HH_\al)$ depends on the choice of $\lambda$, the definition of $\HH_\alpha$ does not. 
We often use the short notation $\DD_\alpha := D(\HH_\alpha)\ .$ 
One has $\sigma_c(\HH_\alpha)=\sigma_{ac}(\HH_\alpha)=[0,\infty)$; $\HH_\alpha$ has  a simple negative eigenvalue $\{ E_\al \}$ for any $ \al \in \RE$ and $\psi_\alpha$ is the corresponding eigenvector. Explicitly 
\[
\begin{aligned}
E_\al= -4 e^{-2(2\pi \al +\ga)}\ , \qquad && \psi_\al (\x) =    \frac{1}{2\pi} K_0 (  2\, e^{-(2\pi \al +\ga)}x) &&\ . 
\end{aligned}
\]
Let us also introduce the quadratic form $\mathcal F_\alpha$ on $L^2(\RE^n)$ with domain and action 
\begin{align*}
D(\mathcal F_\alpha)&=\{\psi\in L^2(\RE^n)\ |\ \exists q\in \CO,\  \phi^\lambda\in H^1(\RE^n):\  \psi=\phi^\lambda + qG^\lambda \}\\
\mathcal F_\alpha(\psi)&= \mathcal F^{\lambda}(\psi)+\Gamma^{\lambda}_\alpha |q|^2\ \ \text{and}\ \ 
\mathcal F^{\lambda}(\psi)=\|\nabla \phi^\lambda\|^2 +\lambda(\|\phi^\lambda\|^2 -\|\psi\|^2)\ 
\end{align*}
It does not depend on $\lambda$, it is symmetric, closed, bounded from below. 
The map $\psi\mapsto \mathcal F_\alpha(\psi)+\lambda\|\psi\|^2$ is positive for every $\lambda>-E_\alpha$ and it coincides with  $\langle\psi, (\HH_\alpha+\lambda)\psi\rangle$  $\forall  \psi\in D(\HH_\alpha)$. This allows to interpretate the form domain  $D(\mathcal F_\alpha)$  as the domain of the square root of the positive self-adjoint operator $\HH_\alpha+\lambda$, so that we make use of the notation
\[
D(\mathcal F_\alpha)= D\big( ( \HH_\alpha+\la)^{\frac{1}{2}} \big)=:\DD_\alpha^{\frac{1}{2}}\ , \ \ \ \ \ \ \lambda>-E_\alpha \ .
\]
Notice that algebraically and topologically one has $D(\mathcal F_\alpha)\cong H^1(\RE^2)\oplus\CO\ $ and the form domain is in a natural way a Hilbert space. By functional calculus we can introduce the scale of Hilbert spaces 
\[
\DD_\alpha^s: = D\big( ( \HH_\alpha+\la)^s \big), \qquad s \in \RE,\ \ \ \la > -E_\al \ .\] $\DD_\alpha^s$ is equipped with the norm  $\|\psi\|_{\DD^s_\alpha} : = \|( \HH_\alpha+\la)^s \psi\|$, equivalent to the graph norm of the operator $( \HH_\alpha+\la)^s $. 
In particular, the spaces $\DD_\alpha^s$ and  $\DD_\alpha^{-s}$ are in duality and 
$$
\DD_\alpha^s \hookrightarrow L^2(\RE^2) \hookrightarrow  \DD_\alpha^{-s}
 $$ 
 is a Hilbert triplet. We denote the duality product by $\langle\cdot,\cdot\rangle_{-s,s}\ .$ In the following we will only consider the case $s=\frac{1}{2} $ and we stress that $\| \psi\|_{\DD_\alpha^{\frac{1}{2}}}\cong \|\phi\|_{H^1}\oplus |q|.$\ Finally, we recall that the fundamental solution $G_\lambda$ is positive, radial, strictly decreasing and moreover it has the following asymptotic behavior (see \cite{AS72}, formulae 9.6.12 and 9.6.13 for the first asymptotic and 9.7.2 for the second)
\begin{align}\label{asymG}
G_{\lambda}=
-\frac{1}{2\pi}\ln(\frac{\sqrt{\lambda}}{2}|\x|) - \frac{\gamma}{2\pi} + \text{o}(1) \qquad  \x\to 0\ ,\qquad \qquad  G_{\lambda}\sim{\frac{1}{\sqrt{8\pi\sqrt\lambda|\x|}}e^{-\sqrt{\lambda} |\x|}} \qquad \x\to \infty \ .
\end{align}

\subsection{The NLS equation with a point interaction}\subsubsection{Well posedness}
We are interested in solutions of the Cauchy problem for the NLS equation
\begin{equation} \label{cauchy}
\left\{\begin{aligned}
i \partial_t\psi (t)  &= {\HH_\alpha} \psi (t) + f(\psi)(t) \\
 \psi (0)  &=  \psi_0\in \DD_\al\ \  \text{or}\ \ \DD_\alpha^{\frac{1}{2}}
  \end{aligned} \right.
\end{equation}
where $f(\psi) =g|\psi|^{p-1}\psi$, and $g=\pm 1$.\\
\noindent
The following theorem collects the known results about well posedness in the energy domain (mild solution) and operator domain (strong solution) for the equation \eqref{cauchy} (see  \cite{CFN21} where a detailed analysis of the well posedness of strong solutions is given, also for the three dimensional case, and \cite{FGI22} where a treatment of the solutions in the energy domain is given).
\begin{teo} 
[Well-Posedness in $ \DD_\alpha^{\frac{1}{2}}$ and  $\DD_\alpha$]
Assume $p> 1$ and  $\psi_0 \in \DD_\alpha^{\frac{1}{2}}$. Then the following properties hold true.\\
1) There exists $T>0$ and a unique weak solution of \eqref{cauchy}
in $C([0,T]; \DD_\alpha^{\frac{1}{2}}) \cap C^1([0,T];\DD_\alpha^{-\frac{1}{2}})$.\\
2) The following blow-up alternative holds. Let the maximal existence time be defined as 
\[T^*=\sup_{T>0}\left\{\psi\in C([0,T], \DD_\alpha^{\frac{1}{2}}))\cap C^1([0,T], \DD_\alpha^{-\frac{1}{2}})\ \ \text{solves mildly}\ \ \eqref{cauchy} \right\};\]
then $$\ \ \ \
\lim_{t\to T^*} \| \psi(t)\|_{\DD_\alpha^{\frac{1}{2}}}<\infty \ \ \Longrightarrow\  \ T^*=\infty .\\
$$
3) $L^2$- mass is conserved along the evolution: $\| \psi(t) \|^2=\| \psi_0\|^2\ \forall t\in [0,T^*)\ .$ \\
4) Energy is conserved along the evolution: $E(\psi(t))=E(\psi_0)\ \ \forall t\in [0,T^*)$\\  where 
\beq\label{energy}
E(\psi)=\frac{1}{2}\mathcal F_\al(\psi) +\frac{g}{p+1}\|\psi \|^{p+1}_{p+1}\qquad \forall\psi\in \DD_\al^{\frac{1}{2}}.
\eeq
5) Let $\psi_0 \in \DD_\al$. Then there exists $T>0$ and a unique strong solution $\psi$ of \eqref{cauchy}
in $C([0,T]; \DD_\alpha) \cap C^1([0,T];L^2(\RE^2))$.\\
6) Let $\psi_0 \in \DD_\al$ and the maximal existence time be defined as 
\[
\tilde T^*=\sup_{T>0}\left\{\psi\in C([0,T],\DD(\HH_{\al}))\cap C^1([0,T], L^2) \ \ \text{solves strongly}\ \ \eqref{cauchy} \right\};\]
then 
$\ \ \ \ \lim_{t\to \tilde T^*} \| \psi(t)\|_{\DD_{\al}}<\infty \ \ \Longrightarrow\  \ \tilde T^*=\infty .$\\
7) $\tilde T^*= T^*\ .$\\
8) $T^*=+\infty$ if $g=+1$ and $p>1$ or if $g=-1$ and $1<p<3\ .$
\end{teo}
In the following we will denote as $T^*(\psi_0)$ the maximal existence time of the solution of \eqref{cauchy}. When $T^*(\psi_0)<+\infty$ we say that the solution $\psi(t)$ corresponding to the initial datum $\psi_0$ blows-up in a finite time (in the future, analogous definition for blow-up in the past). We will omit the dependence of $T^*$ on $\psi_0$ when it is clear from the contest.
\\
\subsubsection{Standing waves}
Recall that a standing wave of \eqref{cauchy} is a solution of the form $\psi(t)=e^{i\omega t}\varphi\ .$ The profile $\varphi $ is a solution of the stationary equation
\begin{equation}\label{stationary}
\HH_\alpha \varphi +\omega\varphi + f(\varphi)=0
\end{equation}
equivalent to $S'_{\omega}(\varphi)=0\ ,$ where the action functional $S_\omega$ is defined as
\begin{align}\label{action}
S_\omega(\varphi)=E(\varphi)+\frac{\omega}{2}\|\varphi\|^2 \quad \quad \forall \varphi \in \DD^{\frac{1}{2}}_\al\ .
\end{align}
The set of ground states of the action $S_\omega$ is defined as 
\begin{align}\label{ground}
\mathcal G=\left\{\varphi_\omega\in \DD_\alpha^{\frac{1}{2}}\ \text{s.t.}\ S_\omega(\varphi_\omega)\leq S_\omega(\varphi)\ \  \forall \varphi \in \DD_\alpha^{\frac{1}{2}}\ \text{satisfying}\ S_\omega'(\varphi)=0\right\}
\end{align}
Recently in \cite{FGI22} and \cite{ABCT22} existence and properties of ground states of the action $S_\omega$ for the case of attractive nonlinearity ($i.e. g=-1$) in \eqref{cauchy}  have been proved by variational methods. In particular, a ground state exists for every $\omega>-E_\alpha$ and if $\varphi_\omega\in \mathcal G$ then
\begin{align*}
d(\omega)=\inf\left\{S_\omega(\varphi)\ \text{s.t.}\ \varphi\in \DD_\al^{\frac{1}{2}},\ \varphi\neq 0,\ N_\omega(\varphi)=0 \right\}=S_\omega(\varphi_\omega)
\end{align*}
where $N_\omega$ is the Nehari functional
\begin{align*}
N_\omega(\varphi) = \mathcal F_\alpha(\varphi)+\omega\|\varphi\|^2 -\|\varphi\|^{p+1}_{p+1} \ .
\end{align*}
The following fact is an immediate consequence of the results in \cite{FGI22} and \cite{ABCT22} and it will be useful later (see also the analogous Lemma in \cite{FO18}).
\begin{prop}\label{variational}
Let  $\varphi_\omega\in \mathcal G$ a ground state of the action $S_\omega$ and $\psi\in \DD_\al^{\frac{1}{2}}\ \text{s.t.}\ \|\psi\|_{p+1}^{p+1}=\|\varphi_\omega\|_{p+1}^{p+1}\ .$ Then 
\begin{itemize}
\item [a)] $N_\omega(\psi)\geq 0$
\item [b)] $S_\omega(\psi)\geq S_\omega(\varphi_\omega).$
\end{itemize}
\end{prop}
\begin{proof} 
From Lemma 3.3 in \cite{FGI22} and $d(\omega)=\inf\{\frac{p-1}{2(p+1)}\|\psi\|^{p+1}_{p+1},\ \ \psi\in N_\omega\}=\frac{p-1}{2(p+1)}\|\varphi_\omega\|_{p+1}^{p+1}=S_\omega(\varphi_\omega)$ property a) follows. Taking into account a) one has $S_\omega(\varphi_\omega)=\frac{p-1}{2(p+1)}\|\varphi_\omega\|^{p+1}_{p+1} =\frac{p-1}{2(p+1)}\|\varphi_\omega\|^{p+1}_{p+1} + \frac{1}{2}N_\omega(\varphi_\omega)\leq\frac{p-1}{2(p+1)}\|\psi\|^{p+1}_{p+1} + \frac{1}{2}N_\omega(\psi)=S_\ome (\psi)\ .$
\end{proof}
Finally, we recall the definition of strong instability of a standing wave.
\begin{dfn}\label{stronginstdef}
The standing wave $\psi(t)=e^{i\omega t}\varphi_\omega\ $ is said to be strongly unstable if for every $\epsilon>0 $ there exist $\psi_0\in \DD^{\frac{1}{2}}_\alpha$ such that $\|\psi_0-\varphi_\omega\|_{\DD_\alpha^{\frac{1}{2}}}<\epsilon$  with $T^*(\psi_0)<+\infty\ .$ 
\end{dfn}
As a last preliminary definition we adapt a classic tool needed in the treatment of the virial functional to the point interaction framework.
 We put\ \
\begin{align}\label{Sigmaal}
\Sigma_\al := \{ \psi \in \DD_\al^{\frac{1}{2}}(\RE^2)\ | \ \x \psi \in L^2(\RE^2) \}.
\end{align}
\noindent

\noindent


\section{Blow-up and strong instability.}
\emph{In the following the value and sign of $\alpha$ will be irrelevant, so we will omit the corresponding pedices in the symbol $\mathcal H_\alpha $, $\mathcal F_\al$,\ $\DD_\alpha\ $, $\DD^{\frac{1}{2}}_\alpha\ $ and $\Sigma_\al$, with the sole exception of Remark \ref{dominiodil}}.
\subsection{Virial identity}\vskip5pt
\begin{lem}
Let  $\psi_0 \in \Sigma$ and $\psi \in C(\left[ 0, T^* \right);\DD^{\frac{1}{2}}(\RE^2))$ the corresponding weak maximal solution of \eqref{cauchy}. Then $ \psi \in C ( \left[ 0, T^* \right); \Sigma).$ Moreover for any fixed $\psi\in \DD^{\frac{1}{2}}(\RE^2)$ the variance
\begin{align*}
 I(t) :=  \int_{\RE^2} \abs{\x}^2 \  \abs{\psi(t,\x)}^2 \ d\x  \ 
\end{align*}
defines a $C^1(\left[0, T^* \right);\RE)$ function and

\begin{align}\label{virial1}
\frac{d}{dt} I(t) = 4 \Im \int_{{\RE^2}} \ \bar{\psi}(t,\x) \x\cdot\nabla_x\psi(t,\x) \ d\x.
\end{align}
\label{lem1sec3}
\end{lem}

\begin{proof}
We firstly show that $t\mapsto \x\psi(t, \x)\in C^0(\left[ 0, T^* \right); L^2(\RE^2))$. Let $ \chi_\epsilon \in \mathcal S(\RE^2)\ ,$ $\chi_\epsilon(x)=e^{-\epsilon|\x|^2}\ $  and define a regularized variance
\[
 I_\epsilon (t) :=  \int_{\RE^2} \abs{\x \chi_\epsilon \psi (t,\x)}^2 \ d\x.
\]
Let $\psi_0 \in \Sigma $ and $ \psi \in C([0,T^*), \DD^{\frac{1}{2}} )\cap C^1(\left[ 0, T^* \right); \DD^{-{\frac{1}{2}}}(\RE^2)) $ the weak solution of the \eqref{cauchy}. One has $\x \chi_\epsilon \psi \in  C([0,T^*), \DD^{\frac{1}{2}}) $  and for any $t\in [0 ,T^*)$ we have
\\
\begin{align*}
& \frac{d}{dt} \ \ \int_{\RE^2} \abs{\x \chi_\epsilon \psi (t,\x)}^2 \ d\x  = 2 \Re \langle |\x|^2 \chi_\epsilon^2 \psi, \partial_t \psi\rangle_{-\frac{1}{2}, \frac{1}{2}} \notag \\
&  = 2\Re \langle |\x|^2 \chi_\epsilon^2 \psi,  -i \HH_\alpha \psi -  ig \abs{\psi}^{p-1}\psi\rangle_{-\frac{1}{2}, \frac{1}{2}} \notag \\
&  = 2\Im \langle \HH(|\x|^2 \chi_\epsilon^2 \psi),  \psi \rangle_{-\frac{1}{2}, \frac{1}{2}} \notag \\& 
= 2\Im \langle -\Delta(|\x|^2 \chi_\epsilon^2 \psi),  \psi \rangle_{-\frac{1}{2}, \frac{1}{2}} \notag \\& =
-2 \Im \int_{\RE^2}\psi {\nabla\cdot\nabla(\overline{ |\x|^2 \chi_\epsilon^2 \psi})}\ \ dx  \notag \\&= -2 \Im  \int_{\RE^2} \psi  \nabla\cdot \left[\chi_\epsilon^2(|\x|^2\nabla \overline\psi +2\x  \overline\psi-2\epsilon{\x} |\x|^2\overline\psi)\right]\ d\x\ .
\end{align*}
Now we can integrate by parts noticing that $\x\nabla\psi\in L_{\text{loc}}^2(\RE^2)$ and after suppressing a real term in the integrand we obtain
\begin{align*}
& \frac{d}{dt} \ \ \int_{\RE^2} \abs{\x \chi_\epsilon \psi (t,\x)}^2 \ d\x =4 \Im  \int \chi_\epsilon^2(1-\epsilon|\x|^2)\overline\psi\ \x\cdot\nabla\psi \ d\x
\end{align*}
and integrating in time
\begin{align}
&  I_\epsilon(t) = I_\epsilon(0) + 4 \Im \int_0^t  \int \chi_\epsilon^2(1-\epsilon|\x|^2)\overline\psi\ \x\cdot\nabla\psi\ d\x \ ds
\label{calcoli2}
\end{align}
Notice now that $\x\cdot\nabla\psi=\x\cdot\nabla\phi^\lambda +q\x\cdot\nabla G^{\lambda}$ and taking into account that $\|\nabla \phi^\lambda\|\leq c\|\psi\|_{\DD^{\frac{1}{2}}}$, $\|\x\cdot q\nabla G^{\lambda}\| \leq c\|\psi\|_{\DD^{\frac{1}{2}}}$ one gets
\begin{align*}
 I_\epsilon(t)\leq I_{\epsilon}(0) +c(m)\int_0^t\|\psi(s)\|_{\frac{1}{2}}\ ds + c\int_0^t \|\psi(s)\|_{\frac{1}{2}}I_{\epsilon}^{\frac{1}{2}}(s)\ ds 
\end{align*}
where $c(m)$ is a constant depending on the mass. From Gr\"onwall inequality it follows that there exists a constant $c$ independent on $\epsilon$ such that
 \begin{align*}
 I_\epsilon(t)\leq c\ \ \ \ \     t\in [0,T^*]
\end{align*}
From Fatou lemma one finally concludes that 

\begin{align*}
 I(t)=\int\lim\inf_\epsilon \chi^2_\epsilon |\x|^2 |\psi(t,\x)|^2\ d\x\leq  \lim\inf_\epsilon\int \chi^2_\epsilon |\x|^2 |\psi(t,\x)|^2\ d\x\leq c\ \ \ \ \     t\in [0,T^*]
\end{align*}
which gives $I(t)\in L^{\infty}\ \ \forall t\in [0,T^*]\ ,$ the map $t\mapsto \||\cdot|u(t,\cdot)\|$ is bounded on any $(0,T)$ with $T<T^*$ and consequently weakly continuous as a map $(0,T^*)\to L^2(\RE^2).$ From \eqref{calcoli2}, the fact that $\overline \psi\x\cdot\nabla\psi\in C_tL^1_x$ and the dominated convergence theorem, we also obtain 
\begin{align*}
&  I(t) = I(0) + 4 \Im \int_0^t  \int \overline\psi\ \x\cdot\nabla\psi\ d\x \ ds\ \ \ \ \ \    \forall t\in [0,T^*]
\end{align*}
which gives at once that the $ \psi \in C^0 ( \left[ 0, T^* \right); \Sigma)\ $ and  validity of \eqref{virial1}.
\end{proof}
\noindent
The crucial information is contained in the following lemma
\begin{lem}\label{lem2sec3}[Virial identity]
Let  $\psi_0 \in \Sigma $ and $\psi \in C(\left[ 0, T^* \right);\DD^{\frac{1}{2}})$ the corresponding maximal weak solution of \eqref{cauchy}. 
Then the function
\[ \ t \mapsto I(t)=  \int_{\RE^2} |\x|^2 \  \abs{\psi(t,x)}^2 \ d\x
\]
is in  $C^2(\left[0, T^* \right);\RE)$ and the following identity holds
\begin{align}
\frac{d^2}{dt^2} I (t) &=  16 E(\psi) +8 g \frac{(p-3)}{p+1}  \|\psi(t)\|^{p+1}_{p+1} +\frac{2}{\pi}|q|^2 \nonumber\\
\label{virial2}
&= 8\mathcal F(\psi) +8 g \frac{(p-1)}{p+1} \|\psi(t)\|^{p+1}_{p+1} +\frac{2}{\pi}|q|^2 = 8Q(\psi)
\end{align}
where 
\beq\label{defQ}
Q(\psi) := \mathcal F(\psi) +g \frac{(p-1)}{p+1} \|\psi(t)\|^{p+1}_{p+1} +\frac{1}{4\pi}|q|^2\ .
\eeq
\end{lem}
\begin{proof}
Let us show the result first assuming that $\psi_0\in \Sigma\cap \DD$ and considering the corresponding strong solution $\psi\in C(\left[ 0, T^* \right);\DD)\cap C^1(\left[ 0, T^* \right);L^2(\RE^2))$. We need to derive in time the r.h.s. of \eqref{virial1}. We regularize it writing
\begin{align}\label{accaeps}
h_{\epsilon}(t):=\Im\int_{\RE^2}e^{-\epsilon|\x|^2}\ \overline \psi\ \x\cdot\nabla\psi\ d\x\ .
\end{align}
Admitting that $\psi\in C^1([0,T^*),\DD^{\frac{1}{2}})\ $ one can safely derive in time \eqref{accaeps}, obtaining 
\begin{align*}
\dot h_{\epsilon}(t):=\Im\int_{\RE^2}e^{-\epsilon|\x|^2}\ \dot{\overline\psi}\ \x\cdot\nabla\psi\ d\x\ + \Im\int_{\RE^2}e^{-\epsilon|\x|^2}\ \overline \psi\ \x\cdot\nabla\dot\psi\ d\x\ .
\end{align*}
Both addenda are well defined, 
and more precisely the map $t\mapsto e^{-\epsilon|\x|^2}\  \x\cdot\nabla\psi$ is in $C^1([0,T^*),L^2(\RE^2))$ because
\begin{align*}
\|e^{-\epsilon|\x|^2}\ \x\cdot\nabla\psi \|=\|e^{-\epsilon|\x|^2}\x\cdot\nabla\phi+ e^{-\epsilon|\x|^2}\ q(t)\x\cdot \nabla G^\lambda \|\leq c_{\epsilon}\left(\|\nabla \phi\|+|q(t)|\right)\leq \|\psi\|_{\DD^{\frac{1}{2}}}\ ,
\end{align*}
Now, integrating by part the second addendum one has 
\begin{align}\label{accaepspunto2}
\dot h_{\epsilon}(t):=\Im\left\{\int_{\RE^2}e^{-\epsilon|\x|^2}\left( \dot{\overline\psi}\ \x\cdot\nabla\psi-\dot\psi\ \x\cdot\nabla\overline \psi\right) d\x\ -2\int_{\RE^2}e^{-\epsilon|\x|^2}\left(\overline \psi\ \dot\psi -\epsilon|\x|^2\overline \psi\dot\psi\right) d\x\ \right\} 
\end{align}
and the r.h.s. is well defined and continuous in time only assuming  $\psi\in C^1([0,T^*), \DD^{\frac{1}{2}})$. By density of $ C^1([0,T^*),\DD^{\frac{1}{2}})\ $ in $ C([0,T^*),\DD^{\frac{1}{2}})\cap C^1([0,T^*),L^2)$ (which is proven as in the case of standard Sobolev case), formula \eqref{accaepspunto2} still holds in these hypotheses. Now we perform a second regularization considering $\psi\in C^0([0,T^*),\DD)\cap C^1([0,T^*),L^2),$ so that we can apply the equation in strong form. From \eqref{accaepspunto2}, by using $\Im z=-\Im\overline z$ and then the equation in \eqref{cauchy} one obtains
\begin{align*}
\dot h_{\epsilon}(t)&=2\Im\int_{\RE^2}e^{-\epsilon|\x|^2} \dot{\overline\psi}\left(\x\cdot\nabla\psi+\psi \right) d\x\ +2\Im \epsilon \int_{\RE^2}e^{-\epsilon|\x|^2}
|\x|^2\overline \psi\dot\psi\ d\x\  \\
&=2\Im  \int_{\RE^2}e^{-\epsilon|\x|^2} i\left(\overline{\HH\psi+f(\psi)}\right)\left(\x\cdot\nabla\psi+\psi \right) d\x\ +2\epsilon\ \Im  (-i)\int_{\RE^2}e^{-\epsilon|\x|^2}
|\x|^2\overline \psi\left(\HH\psi+f(\psi) \right)\ d\x\\
&=2\Re\ \int_{\RE^2}e^{-\epsilon|\x|^2} \left(\overline{\HH\psi+f(\psi)}\right)\psi\ d\x - 2\epsilon \Re\ \int_{\RE^2}e^{-\epsilon|\x|^2} |\x|^2\overline \psi\left({\HH\psi+f(\psi)}\right)\ d\x\ + \\
&\ \ \ \ 2 \Re\int_{\RE^2}e^{-\epsilon|\x|^2} \overline{\HH\psi}\left(\x\cdot\nabla\psi \right)\ d\x + 2 \Re\int_{\RE^2}e^{-\epsilon|\x|^2} \overline{f(\psi)}\left(\x\cdot\nabla \psi \right)\ d\x = I + II + III + IV
\end{align*}
Thanks to the dominated convergence theorem the term $I$ converges to 
\begin{align}
2\left(\langle \HH\psi,\psi \rangle +g \|\psi\|_{p+1}^{p+1}\right)=&  4E(\psi) -4\int_{\RE^2} F(\psi) \ d\x +2(p+1)\int_{\RE^2} F(\psi)\ d\x \nonumber \\
\label{A} =&4E(\psi)+2(p-1)\int_{\RE^2} F(\psi)\ d\x \end{align}
where we have denoted 
\begin{align*}
F(\psi)=\frac{g}{p+1}|\psi|^{p+1}\ .
\end{align*}
The term $II$ is vanishing and now let us consider $III$ and $IV$. To treat $IV$ we make use of the identity
$$
2\Re\ e^{-\epsilon |\x|^2}\overline{f(\psi)}\ \x\cdot\nabla \psi\ =\ \nabla\cdot(2\x e^{-\epsilon |\x|^2}F(\psi)) + 4\epsilon |\x|^2e^{-\epsilon |\x|^2}F(\psi) - 4 e^{-\epsilon |\x|^2}F(\psi)$$
and it follows by the divergence theorem and dominated convergence that
\begin{align}\label{B}
 2 \Re\int_{\RE^2}e^{-\epsilon|\x|^2} \overline{f(\psi)}\x\cdot\nabla\psi\ d\x =  \int_{\RE^2} \left( 4\epsilon |\x|^2e^{-\epsilon |\x|^2}F(\psi)  - 4 e^{-\epsilon |\x|^2}F(\psi) \right)\ d\x \longrightarrow - 4 \int_{\RE^2}F(\psi) \ d\x
 \end{align}
For $III$ we preliminarily decompose the domain element in regular and singular part, obtaining
\begin{align*}
&2\Re \int_{\RE^2} e^{-\epsilon|\x|^2}\overline{\HH\psi}\ \x\cdot\nabla\psi\ d\x=\\ &2\Re \int_{\RE^2} e^{-\epsilon|\x|^2}\left(-\overline{\Delta\phi^\lambda}\ \x\cdot\nabla\phi^\lambda-\lambda\ \overline {q} G^{\lambda}\x\cdot\nabla\phi^\lambda- q\overline{\Delta\phi^\lambda}\x\cdot\nabla G^{\lambda}- \lambda |q|^2 G^{\lambda} \x\cdot\nabla G^{\lambda} \right) d\x  =\\ &III_a +III_b +III_c +III_d
\end{align*}
Now we treat separately the various addenda. Integrating by parts $III_a$ one has
\begin{align*}
III_a=&2\Re \int_{\RE^2} (-\overline{\Delta\phi^\lambda})\ e^{-\epsilon|\x|^2}\x\cdot\nabla\phi^\lambda \ d\x=2\Re\int_{\RE^2} \overline{\nabla\phi^\lambda}\cdot\nabla(e^{-\epsilon|\x|^2} \x\cdot \nabla\phi^\lambda)\ d\x = \\ & 2\Re\int_{\RE^2} \overline{\nabla\phi^\lambda}\cdot\nabla(\x\cdot \nabla\phi^\lambda) e^{-\epsilon|\x|^2}\ d\x -4\epsilon \Re \int_{\RE^2} e^{-\epsilon|\x|^2} |\x\cdot \nabla \phi^\lambda|^2\ d\x
\end{align*}
The second term vanishes by dominated convergence and the first term vanishes as well exploiting the following identity, which holds true in the two dimensional case, 
\begin{align*}
2\Re\ e^{-\epsilon|\x|^2}\overline{\nabla\phi^\lambda}\cdot\nabla(\x\cdot \nabla\phi^\lambda)  =2\epsilon|\x|^2 e^{-\epsilon|\x|^2}|\nabla\phi^\lambda|^2 + \nabla\cdot(\x e^{-\epsilon|\x|^2}|\nabla \phi^\lambda|^2)\ ,
\end{align*}
and then integrating and applying the divergence theorem and dominated convergence again.	\\
To proceed, let us note preliminary the following identities easily obtained by Fourier transform (where formula \eqref{Glambda} is used and it is essential the dimension $2$ in the first):
\begin{align}\label{Fourier}
\mathscr F(\x\cdot\nabla G^\lambda)=&-\frac{1}{2\pi}\nabla\cdot \frac{\k}{(|\k|^2+\lambda)}=-\frac{1}{2\pi} \frac{2\lambda}{(|\k|^2+\lambda)^2} \\
\mathscr F(\nabla\cdot\x G^\lambda)=&-\frac{1}{2\pi}\k\cdot\nabla\frac{1}{|\k|^2+\lambda}=\frac{1}{2\pi} \frac{2|\k|^2}{(|\k|^2+\lambda)^2}
\end{align}
In particular one sees that $\x\cdot\nabla G^\lambda \in H^2(\RE^2)$ and $\x G^{\lambda} \in H^1(\RE^2)$, and we can integrate by parts in
\begin{align*}
III_b +III_c= &2\Re \int_{\RE^2} e^{-\epsilon|\x|^2} \left(-\lambda\ \overline {q} G^{\lambda}\x\cdot\nabla\phi^\lambda- q\overline{\Delta\phi^\lambda}\x\cdot\nabla G^{\lambda}\right) d\x \\ =& 2\Re\int_{\RE^2} \left(\lambda G^\lambda \overline q \phi^\lambda \x\cdot\nabla e^{-\epsilon|x|^2} - q\overline\phi^{\lambda} \x\cdot\nabla G^\lambda  (\Delta e^{-\epsilon|\x|^2}) -2q\overline\phi^{\lambda}\nabla(\x\cdot\nabla G^\lambda)\cdot\nabla e^{-\epsilon|\x|^2} \right) \ d\x \\ +&2\Re\int_{\RE^2} e^{-\epsilon |\x|^2}(\lambda\overline q \phi^\lambda \nabla\cdot(\x G^\lambda) -q\overline \phi^\lambda\Delta(\x\cdot\nabla G^\lambda))\ d\x
\end{align*}
The last integral identically vanishes and from anyone of the terms in the first integral can be extracted a factor $\epsilon$; one concludes that $III_b +III_c \longrightarrow 0$ by dominated convergence.\\
It remains to consider the limit for $\epsilon \to 0$ of $III_d$, which can be computed explicitly thanks to the Plancherel theorem and identities \eqref{Glambda} and \eqref{Fourier}:
\begin{align}
III_d=&2\Re \int_{\RE^2} e^{-\epsilon|\x|^2}\left(-\lambda |q|^2 G^{\lambda} \x\cdot\nabla G^{\lambda} \right) d\x  \longrightarrow -2\lambda |q|^2 \Re \int_{\RE^2}  G^{\lambda} \x\cdot\nabla G^{\lambda} d\x \nonumber \\
=& \frac{2\lambda^2|q|^2}{\pi}\int_0^\infty\frac{r}{(r^2+\lambda)^3}\ dr=\frac{1}{2\pi}|q|^2\label{C} .
\end{align}
Finally, collecting \eqref{A},\eqref{B}, \eqref{C} and taking into account that the other terms involved vanish, we obtain
\begin{align*}
\ddot I(t) = & 4\lim_{\epsilon\to 0}\dot h_\epsilon (t)= 16 E(\psi)  +(8p-24)\int_{\RE^2}F(\psi)\ d\x + \frac{4}{\pi}|q|^2\\
=&16 E(\psi) + 8g\frac{p-3}{p+1}\int_{\RE^2}|\psi|^{p+1}\ d\x + \frac{2}{\pi}|q|^2\\
=& 8\mathcal F(\psi) + 8g \frac{(p-1)}{p+1} \|\psi(t)\|^{p+1}_{p+1} +\frac{2}{\pi}|q|^2 \end{align*}
Having proved the identity \eqref{virial2} for strong solutions, the same identity follows for weak solutions exploiting continuous dependence and density,
and this ends the proof of the Lemma.
\end{proof}
\subsection{Mass preserving scaling and its properties}
From now on, we will only consider the attractive nonlinearity, i.e. the case $g=-1$.
\begin{dfn}
Let us introduce the mass preserving scaling map $T^\sigma:L^2(\RE^2)\rightarrow L^2(\RE^2)$
\begin{align*}
T^\sigma(\psi)\equiv\psi^{\sigma}(\x)=\sigma\psi(\sigma \x)\qquad\ \ \forall \psi \in L^2(\RE^2)
\end{align*}
\end{dfn}
\noindent
\begin{remark}
By using 
$G^{\lambda}(\sigma\x)=\frac{1}{2\pi}K_0(\sqrt\lambda \sigma|x|)=\frac{1}{2\pi}K_0(\sqrt{\lambda\sigma^2}|x|)=G^{\lambda\sigma^2}(\x):=G^{\lambda^\sigma}(\x)\ $ one obtains
\begin{equation*}
\psi^\sigma(\x)=\sigma \phi(\sigma\x)+q \sigma G^\lambda(\sigma\x)=\sigma \phi(\sigma\x)+ \sigma qG^{\lambda\sigma^2}(\x)=\phi^\sigma(\x) +q^\sigma G^{\lambda^\sigma}(\x)
\end{equation*}
and the map $\psi\to \psi^\sigma$ leaves invariant $\DD^{\frac{1}{2}}\ $ (with the same $\alpha$). It also follows that $q^\sigma=\sigma q\ .$\\ 
\end{remark}
\begin{remark}\label{dominiodil} 
One has $G^\lambda-G^{\lambda\sigma^2}\in H^{3-\epsilon}\ \ \forall \epsilon>0$ 
and exploiting the first asymptotic relation in \eqref{asymG}, one obtains $(G^{\lambda\sigma^2}-G^\lambda)(\0)=-\frac{1}{2\pi}\log{\sigma}$. From this one concludes that $T^\sigma$ does not preserve $\DD_\al\ .$ Instead, $T^\sigma:\DD_\alpha\rightarrow \DD_{\al -\frac{1}{2\pi}\log{\sigma}} .$ In fact, from 
$\psi= \phi^\la   +q \,  G^\lambda ,\, \phi^\la(\0)=\Gamma^\lambda_\alpha q $, it follows 
\begin{align*}\psi^\sigma(\x)&= \sigma\phi^\la(\sigma\x)   +\sigma q \,  G^\lambda(\sigma\x)=\sigma\phi^\la(\sigma\x) + \sigma q(G^{\lambda\sigma^2}-G^\lambda)(\sigma\x) + \sigma q G^\lambda(\x)\\&=(\phi^\lambda)^\sigma(\x) + q^\sigma G^\lambda(\x)
\end{align*}
and $(\phi^\la)^\sigma(\0)=\sigma\phi^\la(\0)- \sigma q\frac{1}{2\pi}\log{\sigma} = \sigma q(\Gamma^\lambda_\alpha - \frac{1}{2\pi}\log\sigma)=q^\sigma\Gamma^\lambda_{\alpha- \frac{1}{2\pi}\log\sigma}\ ,$ hence $\psi^\sigma\in \DD_{\al -\frac{1}{2\pi}\log{\sigma}}\ .$
\end{remark}

\begin{prop}\label{Tsigma1} Let $\psi \in \DD^{\frac{1}{2}}$. Then
\begin{align*}
\mathcal F(\psi^{\sigma})&=\sigma^2\mathcal F(\psi) +\frac{|q|^2}{2\pi}\sigma^2\log\sigma \ \\
 \|\psi^{\sigma}\|^{p+1}_{p+1}&=\sigma^{p-1} \|\psi\|^{p+1}_{p+1}\\ 
\frac{d}{d\sigma}\mathcal F(\psi^{\sigma})\vert_{\sigma=1}&=2\mathcal F(\psi) +\frac{1}{2\pi}|q|^2\qquad\qquad   \\
\frac{d}{d\sigma} \|\psi^{\sigma}\|^{p+1}_{p+1}\vert_{\sigma=1}&=(p-1) \|\psi\|^{p+1}_{p+1}
\qquad \qquad \quad\\
\frac{d}{d\sigma}E(\psi^\sigma)\vert_{\sigma=1}&=\frac{d}{d\sigma}S(\psi^\sigma)\vert_{\sigma=1}=Q(\psi) \ .
\end{align*}
 
\end{prop}
\begin{proof}
From $\psi^\sigma=\sigma \phi(\sigma\x)+ \sigma qG^{\lambda\sigma^2}$ and the identity $\Gamma^{\lambda\sigma^2}=\Gamma^\lambda + \frac{1}{2\pi}\log\sigma $ one obtains immediately $\mathcal F(\psi^{\sigma})=\sigma^2 \| \nabla \phi\|^2 
 +\sigma^2\Gamma^{\lambda\sigma^2}_\alpha |q|^2
\ +\lambda\sigma^2(\|\phi\|^2 -\|\psi\|^2)=\sigma^2\mathcal F(\psi) +\frac{|q|^2}{2\pi}\sigma^2\log\sigma \ .$
The other identities are obtained by direct computation without difficulty.
\end{proof}
For fixed $\varphi_\omega\in\mathcal G$ let us define the functions $\sigma\mapsto {S_\omega(\varphi_\omega^{\sigma})} $ and $\sigma\mapsto {Q_\omega(\varphi_\omega^{\sigma})}$ given by
\begin{align}\label{Ssigma}
S_\omega(\varphi_\omega^{\sigma})&=\frac{\sigma^2}{2}\mathcal F(\varphi_\omega) + \frac{\omega}{2}\|\varphi_\omega\|^2+\frac{\sigma^2}{4\pi}\log\sigma |q_{\omega}|^2 -\frac{1}{p+1}\sigma^{p-1}\|\varphi_\omega\|_{p+1}^{p+1} \\
Q(\varphi^\sigma_\omega)&=\sigma^2\mathcal F(\varphi_\omega) +\frac{\sigma^2}{2\pi}\log\sigma |q_{\omega}|^2 -\frac{p-1}{p+1}\sigma^{p-1}\|\varphi_\omega\|_{p+1}^{p+1} +\frac{\sigma^2}{4\pi}|q_\omega|^2
\end{align} 
It is immediate that the functions $\sigma\mapsto S_\omega(\varphi_\omega^{\sigma})$ and $\sigma\mapsto Q(\varphi^\sigma_\omega)$ belong to $ \in C^\infty (\RE^+)$.\\ Let us now denote, again for fixed $\varphi_\omega \in \mathcal G$,
\beq\label{ABC}
A=\mathcal F(\varphi_\omega)+\frac{1}{4\pi}|q_\omega|^2,\  B=\frac{1}{2\pi}|q_\omega|^2,\  C=\frac{p-1}{p+1}\|\varphi_\omega\|_{p+1}^{p+1}\ .
\eeq  
Then we have
\begin{prop}\label{derivateS} Let 
$\varphi_\omega\in \mathcal G$. Then the following identities hold:
\begin{align}
\label{S'}\frac{d}{d\sigma}S_\omega(\varphi_\omega^{\sigma})=&A\sigma  + B\sigma\log\sigma - C\sigma^{p-2}\\
\label{S''}\frac{d^2}{d\sigma^2}S_\omega(\varphi_\omega^{\sigma})=&(A+B)+B\log\sigma-C(p-2)\sigma^{p-3}\\
\label{S'''}\frac{d^3}{d\sigma^3}S_\omega(\varphi_\omega^{\sigma})=&\frac{B}{\sigma}-C(p-2)(p-3)\sigma^{p-4}\\
\label{A=C}\frac{d}{d\sigma}S_\omega(\varphi_\omega^{\sigma})\vert_{\sigma=1}=&0 \ \text{or equivalently}\ \ A=C\\
\label{Q=0} Q(\varphi_\omega)=&0\\
\label{Qsigma}Q(\varphi^\sigma_\omega)=&\sigma\frac{d}{d\sigma}S_\omega(\varphi_\omega^\sigma)\\ 
\label{Q'} 
\frac{d}{d\sigma}Q(\varphi_\omega^\sigma)=&\frac{d}{d\sigma}S_\omega(\varphi_\omega^{\sigma})+\sigma\frac{d^2}{d\sigma^2}S_\omega(\varphi_\omega^{\sigma})\ .
\end{align}
\end{prop}
\begin{proof}
The proof of \eqref{S'}, \eqref{S''}, \eqref{S'''}) is obtained by direct computation of the derivatives taking into account (3.12) and (3.14). Property \eqref{A=C} is obtained just exploiting $\varphi_\omega\in \mathcal G$ i.e. $S'_\omega(\varphi_\omega)=0\ .$ Property \eqref{Q=0} is a reformulation of \eqref{A=C}. Properties \eqref{Qsigma} and \eqref{Q'} are based on the previously proven identities. For \eqref{Qsigma},
\begin{align*}
\sigma\frac{d}{d\sigma}S_\omega(\varphi_\omega^\sigma)&=\sigma^2\mathcal F(\varphi_\omega) + \frac{\sigma^2}{2\pi}\log\sigma |q_{\varphi_\omega}|^2 + \frac{\sigma^2}{4\pi} |q_{\varphi_\omega}|^2 -\frac{p-1}{p+1}\sigma^{p-1}\|\varphi_\omega\|_{p+1}^{p+1}=Q(\varphi^\sigma_\omega)\ .
\end{align*}
Identity \eqref{Q'} is obtained by deriving \eqref{Qsigma}.
\end{proof}
\begin{remark} In the previous proposition, properties \eqref{S'}, \eqref{S''}, \eqref{S'''}, \eqref{Qsigma}, \eqref{Q'} do not depend from $\varphi_\omega$ being a stationary state and they hold for every $\varphi\in \DD^{\frac{1}{2}}\ .$
\end{remark}
\subsection{Blow-up and strong instability}
\vskip3pt\noindent The next result is crucial for the analysis.
\begin{lem}\label{mainlemma1} 
Let $p>3$, $\varphi\in \DD^{\frac{1}{2}}\ ,$ $\varphi\neq 0$, $E(\varphi)\geq 0\ ,$ $Q(\varphi)\leq 0$ and $\varphi_\omega\in \mathcal G$; then
\begin{align*}
  S_{\omega}(\varphi_\omega)<S_\omega(\varphi)-\frac{1}{2}Q(\varphi)\ .
\end{align*}
\end{lem}
\begin{proof}
Let  
\[
\sigma_0=\left(\frac{\|\varphi_{\omega}\|_{p+1}^{p+1}}{\|\varphi\|_{p+1}^{p+1}} \right)^{\frac{1}{p-1}}\ .
\]
Then $\|\varphi^{\sigma_0}\|_{p+1}=\|\varphi_\omega\|_{p+1}$ and thanks to Lemma \ref{variational} it follows $S_\omega(\varphi_\omega)\leq S_\omega(\varphi^{\sigma_0})\ .$  Now consider the real function
\begin{align*}
g(\sigma):=S_\omega(\varphi^\sigma)-\frac{\sigma^2}{2}Q(\varphi)=\frac{\omega}{2}\|\varphi\|^2+\frac{\sigma^2}{4\pi}(\log\sigma -\frac{1}{2})|q_\varphi|^2 -\frac{\sigma^2}{p+1}\left(\sigma^{p-3}-\frac{p-1}{2}\right)\|\varphi\|_{p+1}^{p+1}
\end{align*}
Suppose that $g(\sigma_0)\leq g(1);$ then, from the variational characterization \ref{variational} of $\varphi_\omega$ and $Q(\varphi)\leq 0$ it follows
\begin{align*}
S_\omega(\varphi_\omega)\leq S_\omega(\varphi^{\sigma_0})\leq S_\omega(\varphi^{\sigma_0}) -\frac{\sigma_0^2}{2}Q(\varphi)\leq  S_\omega(\varphi) -\frac{1}{2}Q(\varphi) 
\end{align*}
which is the thesis. So it is enough to show that $g(\sigma_0)\leq g(1)\ .$ Actually we will show that $\sigma=1$ is the unique point of absolute maximum of $g$.
One has 
\begin{align*}
g'(\sigma)=B\sigma\log\sigma-A\sigma(\sigma^{p-3} -1)\ .
\end{align*}
It is immediate that $\sigma=1$ is a root. An elementary analysis shows that a second root $\sigma^*$ exists in $(0,1]$. It is an easy check that in $\sigma=1$ there is a maximum and in $\sigma^*\in (0,1)$ there is a minimum whatever are $A$ and $B$. Moreover, being $Q(\varphi)\leq 0$ and  $E(\varphi)>0$, one has that 
\begin{align*}
g(1)=S_\omega(\varphi) -\frac{1}{2}Q(\varphi)\geq S_\omega(\varphi)\geq \frac{\omega}{2}\|\varphi\|^2=g(0^+)\ .
\end{align*}
Finally, thanks to $p>3$ one has $\lim_{\sigma\to +\infty}g(\sigma)=-\infty$ and this ends the proof.
\end{proof}
\begin{proof}[Proof of Theorem \ref{blowup1}]
Let us set 
\begin{align*}
U_\omega=\left\{\varphi\in \DD^{\frac{1}{2}}(\RE^2)\  \text{s.t.}\ S_\omega(\varphi)<S_\omega(\varphi_\omega), \  E(\varphi)\geq 0, \ Q(\varphi)<0  \right\}.
\end{align*}
 We firstly show that the set $U_\omega$ is invariant for the flow of \eqref{cauchy}. Let $\psi_0\in U_\omega$ and $\psi(t)$ the corresponding weak solution of \eqref{cauchy}. Thanks to the conservation law of mass and energy, one has that 
$S_\omega(\psi(t))<S_\omega(\varphi_\omega)\ \text{and}\ E(\psi(t))\geq 0 \ \forall t\in (0, T^*)$. It remains to show that $Q(\psi(t))<0\ .$ Suppose, by absurd that there exist a time $\overline t\in (0, T^*)$ such that $Q(\psi(\overline t))=0$. Being necessarily $\psi(\overline t)\neq 0$, applying Lemma \ref{mainlemma1} one obtains
\begin{align*}
S_\omega(\varphi_\omega)<S_\omega(\psi(\overline t))-\frac{1}{2}Q(\psi(\overline t))=S_\omega(\psi(\overline t))
\end{align*}
against the hypotheses. So $Q(\psi(t))<0\ \forall t\in (0, T^*)\ .$ Now let $\psi_0\in U_\omega\cap \Sigma$, it follows from Lemma \ref{lem1sec3} and the invariance of $U_\omega $ that the solution $\psi(t)\in U_\omega\cap \Sigma\ \forall t\in (0,T^*)\ .$ From Lemma \ref{lem2sec3} and in particular \ref{virial2}, exploiting conservation laws of mass and energy, it follows that 
\begin{align*}
\frac{1}{8}\frac{d^2}{dt^2} I_\psi (t) = Q(\psi(t))<2(S_\omega(\psi(t))-S_{\omega}(\varphi_\omega))=2(S_\omega(\psi_0)-S_{\omega}(\varphi_\omega))<0 \ \ \forall t\in (0, T^*(\psi_0))
\end{align*}
and this implies $T^*(\psi_0)<+\infty\ $ by the classical elementary concavity estimate.
\end{proof}
\begin{proof}[Proof of Theorem \ref{stronginst1}]
By elliptic regularity it follows that $\varphi_\omega\in \Sigma\ .$  Now consider $\varphi_\omega^{\sigma}(\x)=\sigma\varphi_\omega(\sigma\x)\in \Sigma\ .$ Notice that from formula \eqref{ABC} and formula \eqref{A=C}
\begin{align*}
E(\varphi_\omega)>0\ \iff\ \frac{1}{2}(A-\frac{B}{2})-\frac{1}{p-1}C>0 \iff\ \frac{1}{2\pi}|q_\omega|^2<2\frac{p-3}{p+1}\|\varphi_\omega\|_{p+1}^{p+1}
\end{align*}
As already known,  $\sigma=1$ is a stationary point of 
$\sigma\mapsto S_\omega(\varphi_\omega^\sigma) ,$
Moreover, $\frac{d^2}{d\sigma^2}S_\omega(\varphi_\omega^\sigma)\vert_{\sigma=1}=\{(A+B)+B\log\sigma-A(p-2)\sigma^{p-3}\}\vert_{\sigma=1}=A(3-p) +B$ so that 
\begin{align*}
\frac{d^2}{d\sigma^2}S_\omega(\varphi_\omega^{\sigma})\vert_{\sigma=1}<0 \ \iff B<(p-3)A \ \iff \ \frac{1}{2\pi}|q_\omega|^2< \frac{(p-1)(p-3)}{p+1}\|\varphi_\omega\|_{p+1}^{p+1}
\end{align*}
This means that $p>3$ and $E(\varphi_\omega)>0$ imply that $\sigma=1$ is a local maximum for $\sigma\mapsto S_\omega(\varphi_\omega^\sigma) $ and actually the absolute maximum, thanks to $S(\varphi_\omega)>\omega\|\varphi_\omega\|^2=S(\varphi_\omega^\sigma)\vert_{0^+}$.
 Consequently  $S_\omega(\varphi_\omega^\sigma)<S_\omega(\varphi_\omega)\ \ \forall \sigma>1\ .$ Finally from formula \eqref{Qsigma} 
and $\sigma>1$ one has
\begin{align*}
Q(\varphi^\sigma_\omega)=\sigma\frac{d}{d\sigma}S_\omega(\varphi_\omega^\sigma)<0\ .
\end{align*}
To summarize, $\varphi^\sigma_\omega\in U_\omega\cap\Sigma \ \forall \sigma>1 .$ Being $\|\varphi^\sigma_\omega-\varphi_\omega\|_{\DD^{\frac{1}{2}}}\to 0$ as $\sigma\to 1$ the proof is complete.
\end{proof}
\begin{remark}
\label{omegagrande}
The condition  $E(\varphi_\omega)>0\ $ is expected to be true for $\omega>\omega^*$ great enough. That this should be true can be understood by means of the scaling $\varphi_{\omega}(\x)\rightarrow \hat\varphi_\omega(\x)=\omega^{-\frac{1}{p-1}}\varphi_{\omega}(\frac{\x}{\sqrt \omega})\ .$ One has
\beq
\label{scaled}
\mathcal H_{\hat \alpha}\hat\varphi_\omega + \hat\varphi_\omega - |\hat\varphi_\omega|^{p-1}|\hat\varphi_\omega|=0
 \eeq
 with the modified parameter $\hat\alpha=\alpha + \frac{1}{4\pi}\log\omega.$ Formally, $\hat\alpha\to +\infty \ \text{as}\ \omega\to \infty$ and the operator $\mathcal H_{\hat \alpha} \rightarrow -\Delta$ so that \eqref{scaled} reduces to the standard NLS, for which it is well known that the ground state has positive energy if $p>3$ (see for example Corollary 8.1.3 in \cite{caz}). The previous formal argument works rigorously for fairly general Schr\"odinger operators $-\Delta+V$ (see \cite{FO03}).
\end{remark}

\section{Strong instability with $\frac{d^2}{d\sigma^2}S_\omega(\varphi_\omega^\sigma)\vert_{\sigma=1}\leq0 $.}
It appears from the proof of the previous result that the condition $E(\varphi_\omega)\geq 0$  is in general stronger than the condition $\frac{d^2}{d\sigma^2}S_\omega(\varphi_\omega^\sigma)\vert_{\sigma=1}\leq0$. So is a natural generalization of the result given in the previous section  consists in  assuming $\frac{d^2}{d\sigma^2}S_\omega(\varphi_\omega^\sigma)\vert_{\sigma=1}\leq0\  $ as the condition selecting the frequencies of the ground waves the instability of which we want to prove. This more general condition has been advocated by M. Ohta in several papers with various collaborators (\cite{FO03, GO20}, see also \cite{O19}).
\begin{dfn}
Let $\varphi_\omega\in \mathcal G$ and set 
\begin{align*}
V_\omega=\left\{\varphi\in \DD^{\frac{1}{2}}(\RE^2)\  \text{s.t.}\ S_\omega(\varphi)<S_\omega(\varphi_\omega),  \ Q(\varphi)<0, \ \|\varphi\|\leq \|\varphi_\omega \|, 	 \|\varphi\|_{p+1}> \|\varphi_\omega\|_{p+1} \right\}\ .
\end{align*}
\end{dfn}

\begin{lem}\label{mainlemma2} 
Let $p>3$, $\varphi_\omega\in \mathcal G$ with $\omega$ s.t. $\frac{d^2}{d\sigma^2}S_\omega(\varphi_\omega)\leq0$, and let $\varphi\in \DD^{\frac{1}{2}}\ $ such that $\varphi\neq 0$, 
$Q(\varphi)\leq 0$, $\|\varphi\|\leq \|\varphi_\omega \|$, 	 $\|\varphi\|_{p+1}> \|\varphi_\omega\|_{p+1} $ ; then
\begin{align*}
  S_{\omega}(\varphi_\omega)<S_\omega(\varphi)-\frac{1}{2}Q(\varphi)\ .
\end{align*}
\end{lem}
\begin{proof}
Let  
\[
\sigma_0=\left(\frac{\|\varphi_{\omega}\|_{p+1}^{p+1}}{\|\varphi\|_{p+1}^{p+1}} \right)^{\frac{1}{p-1}}\ .
\]
Then $\sigma_0\in (0,1]$, $\|\varphi^{\sigma_0}\|_{p+1}=\|\varphi_\omega\|_{p+1}=\sigma_0^{\frac{p-1}{p+1}}\|\varphi\|_{p+1}$.  Now consider the real function
\begin{align*}
g(\sigma):=S_\omega(\varphi^\sigma)-\frac{\sigma^2}{2}Q(\varphi)=\frac{\omega}{2}\|\varphi\|^2+\frac{\sigma^2}{4\pi}(\log\sigma -\frac{1}{2})|q_\varphi|^2 -\frac{\sigma^2}{p+1}(\sigma^{p-3}-\frac{p-1}{2})\|\varphi\|_{p+1}^{p+1}
\end{align*}
Suppose that $g(\sigma_0)\leq g(1)\ ;$ then, thanks to lemma \ref{variational} it follows $S_\omega(\varphi_\omega)\leq S_\omega(\varphi^{\sigma_0})$ and being $Q(\varphi)\leq 0$ one has
\begin{align*}
S_\omega(\varphi_\omega)\leq S_\omega(\varphi^{\sigma_0})\leq S_\omega(\varphi^{\sigma_0}) -\frac{\sigma_0^2}{2}Q(\varphi)\leq  S_\omega(\varphi) -\frac{1}{2}Q(\varphi) 
\end{align*}
which is the thesis. So it is enough to show that $g(\sigma_0)\leq g(1)\ .$ This inequality is equivalent to
\begin{align*}
\frac{\sigma_0^2}{4\pi}(\log\sigma_0 -\frac{1}{2})|q_\varphi|^2 -\frac{\sigma_0^2}{p+1}\left(\sigma_0^{p-3}-\frac{p-1}{2}\right)\|\varphi\|_{p+1}^{p+1}\leq-\frac{1}{8\pi}|q_\varphi|^2 +\frac{p-3}{2(p+1)}\|\varphi\|_{p+1}^{p+1}
\end{align*}
\begin{align*}
(\frac{\sigma_0^2\log\sigma_0}{4\pi} -\frac{\sigma_0^2}{8\pi}+\frac{1}{8\pi})|q_\varphi|^2 \leq \frac{1}{p+1}\left(\sigma_0^{p-1}-\frac{(p-1)}{2}\sigma_0^2+\frac{p-3}{2}\right)\|\varphi\|_{p+1}^{p+1}
\end{align*}
or also
\begin{align*}
|q_\varphi|^2 \leq \frac{4\pi}{p+1}\frac{2\sigma_0^{p-1}-(p-1)\sigma_0^2+(p-3)}{2\sigma_0^2\log\sigma_0 -\sigma_0^2+1}\|\varphi\|_{p+1}^{p+1}\ .
\end{align*}
The idea is to find an estimate of the type $|q_\varphi|^2 \leq h(\sigma_0)\|\varphi\|_{p+1}^{p+1}$ by making use of the hypotheses on $\varphi$, and then to verify that 
\[
h(\sigma_0)\leq  \frac{4\pi}{p+1}\frac{2\sigma_0^{p-1}-(p-1)\sigma_0^2+(p-3)}{2\sigma_0^2\log\sigma_0 -\sigma_0^2+1}\ \ \ \forall \sigma \in (0,1)
\]
that would prove $g(\sigma_0)\leq g(1)\ .$\\
Notice that from \eqref{S''} and \eqref{A=C}, 
\begin{align}
\frac{d^2}{d\sigma^2}S_\omega(\varphi_\omega^{\sigma})\vert_{\sigma=1}\leq 0\ &\iff \ \mathcal F(\varphi_\omega)+\frac{3}{4\pi}|q_\omega|^2 \leq \frac{(p-1)(p-2)}{p+1}\|\varphi_\omega\|_{p+1}^{p+1} \nonumber\\ 
&\iff \frac{1}{2\pi}|q_\omega|^2 \leq \frac{(p-3)(p-1)}{p+1}\|\varphi_\omega\|_{p+1}^{p+1} \label{aggiunta}
\end{align}
The following Pohozaev identity is obtained applying the computations done in Lemma \ref{lem2sec3} to the stationary equation \eqref{stationary}, or equivalently combining the constraints $N_\omega(\varphi_\omega)=0$ and $Q(\varphi_\omega)=0$:
\begin{align*}
\omega\|\varphi_\omega\|^2=\frac{|q_\omega|^2}{4\pi}+\frac{2}{p+1}\|\varphi_\omega\|_{p+1}^{p+1}
\end{align*}
and making use of inequality \eqref{aggiunta} one gets
\begin{align*}
\omega\|\varphi_\omega\|^2 = \frac{|q_\omega|^2}{4\pi}+\frac{2}{p+1}\|\varphi_\omega\|_{p+1}^{p+1}\leq \left( \frac{2}{p+1} +\frac{1}{2}\frac{(p-3)(p-1)}{p+1} \right)\|\varphi_\omega\|_{p+1}^{p+1}= \frac{p^2-4p+7}{2(p+1)}\|\varphi_\omega\|_{p+1}^{p+1}
\end{align*}
Now, from $\|\varphi\|\leq \|\varphi_\omega \|$, 	 $\sigma_0^{p-1}\|\varphi\|_{p+1}=\|\varphi_\omega\|_{p+1} $ one obtains
\begin{align}\label{omegaphi2}
\omega\|\varphi\|^2 \leq \frac{p^2-4p+7}{2(p+1)}\sigma_0^{p-1}\|\varphi\|_{p+1}^{p+1}\ .
\end{align}
The condition $N(\varphi^{\sigma_0})\geq 0$ which holds thanks to proposition \ref{variational} is equivalent to
\begin{align*}
\sigma_0^2\mathcal F(\varphi)+\frac{1}{2\pi}\sigma_0^2\log\sigma_0 |q_{\varphi}|^2+\omega|\varphi|^2 -\sigma_0^{p-1}\|\varphi\|^{p+1}_{p-1} \geq 0
\end{align*}
and exploiting $Q(\varphi)< 0$ one arrives to
\begin{align*}
-\frac{1}{2\pi}\sigma_0^2\log\sigma_0 |q_\varphi|^2< \sigma_0^2\frac{p-1}{p+1}\|\varphi\|^{p+1}_{p-1}-\frac{\sigma_0^2}{4\pi}|q_{\varphi}|^2  +\omega|\varphi|^2 -\sigma_0^{p-1}\|\varphi\|^{p+1}_{p-1} 
\end{align*}
that combined with \eqref{omegaphi2} yields
\begin{align*}
\frac{1}{4\pi}\sigma_0^2(1-2\log\sigma_0) |q_{\varphi}|^2< \left(\sigma_0^2\frac{p-1}{p+1} + \frac{p^2-4p+7}{2(p+1)}\sigma_0^{p-1} -\sigma_0^{p-1}\right)\|\varphi\|^{p+1}_{p-1} 
\end{align*}
or also
\begin{align*}
 |q_{\varphi}|^2< \frac{4\pi}{p+1} \frac{\left(2(p-1)+ (p^2-4p+7)\sigma_0^{p-3} -2(p+1)\sigma_0^{p-3}\right)}{2(1-2\log\sigma_0)}\|\varphi\|^{p+1}_{p-1} 
\end{align*}
so that the inequality that has to be checked is 
\begin{align*}
\frac{2(p-1)+ (p^2-4p+7)\sigma_0^{p-3} -2(p+1)\sigma_0^{p-3}}{2(1-2\log\sigma_0)}\leq  \frac{2\sigma_0^{p-1}-(p-1)\sigma_0^2+(p-3)}{2\sigma_0^2\log\sigma_0 -\sigma_0^2+1}\ .
\end{align*}
The previous inequality is equivalent to $f(\sigma_0)\geq 0$ where
\begin{align*}
f(\sigma)=(p-3)^2\left(\sigma^{p-1}- 2\sigma^{p-1}\log\sigma \right) -4(p-3)\log\sigma - 4 - (p^2-6p+5)\sigma^{p-3}
\end{align*}
Notice that $f(0^+)=+\infty,\ \ f(1)=0,$ so that to prove the inequality it is sufficient to prove that $f$ is decreasing. That this is indeed the case is a lengthy but elementary check based on the analysis of the derivatives of the function $f$ until the third included. We omit the details.
\end{proof}
\begin{teo}\label{blowup2}
Let $g = -1 $, $p > 3$ and $\psi_0 \in \Sigma \cap V_\omega\ .$   Then $T^*(\psi_0) < +\infty\ .$\\
\end{teo}
\begin{proof}
It is already known that $\psi(t) \in \Sigma$ and by the conservation laws that $S_\omega(\psi(t))<S_\omega(\varphi_\omega)\ $ and $\|\psi(t)\|\leq \|\varphi_\omega \| .$ Thanks to Proposition \ref{variational} and $S_\omega(\psi(t))<S_\omega(\varphi_\omega)\ $ necessarily $\|\psi(t)\|_{p+1}\neq\|\varphi_\omega\|_{p+1} \ \forall t\in(0, T^*(\psi_0))\ ;$ being $\|\psi_0\|_{p+1}> \|\varphi_\omega\|_{p+1}$,  by continuity $\|\psi(t)\|_{p+1}> \|\varphi_\omega\|_{p+1} \ \forall t\in(0, T^*(\psi_0))\ .$ Finally, $Q(\psi(t))<0$ is a consequence of Lemma \ref{mainlemma2}. Now, from \ref{mainlemma2} and the virial identity \ref{virial2} one gets
\begin{align*}
\frac{1}{8}\frac{d^2}{dt^2} I_\psi (t) = Q(\psi(t))<2(S_\omega(\psi(t))-S_{\omega}(\varphi_\omega))=2(S_\omega(\psi_0)-S_{\omega}(\varphi_\omega))<0 \ \ \ \forall t\in (0, T^*)\ .
\end{align*}
This gives the thesis by the classical concavity argument.
\end{proof}

Now we will consider the standing waves and we will show that they are strongly unstable.
\begin{teo}
\label{stronginst2}
Let $g = -1 $ and $p > 3$. Let $\omega> -E_\alpha$ and $\varphi_\omega \in \mathcal G$ such that $\frac{d^2 S(\varphi_\omega^\sigma)}{d \sigma^2}\vert_{\sigma=1}\leq 0\ .$ Then the standing wave $\varphi_\omega e^{i\omega t}$ is strongly unstable.\\
\end{teo}
\begin{proof} One has $\|\varphi^{\sigma}_\omega\|=\|\varphi_\omega\|,\ \|\varphi^{\sigma}_\omega\|_{p+1}=\sigma^{\frac{p-1}{p+1}}\|\varphi_\omega\|_{p+1}>\|\varphi_\omega\|_{p+1} \ \forall \sigma>1\ .$ Now consider the function $S(\varphi^\sigma_\omega)$ given in \ref{Ssigma}. We want to show that $S_\omega(\varphi^\sigma_\omega)<S_\omega(\varphi_\omega)\ \forall \sigma >1\ .$ Thanks to \eqref{S'''} in Proposition \ref{derivateS} we have $\frac{d^3}{d\sigma^3}S_\omega(\varphi_\omega^{\sigma})<0 \ \forall \sigma >1$ from which we deduce that $ \frac{d^2}{d\sigma^2}S_\omega(\varphi_\omega^{\sigma})$ is decreasing for $\sigma>1$. Exploiting the hypothesis $\frac{d^2 S(\varphi_\omega^\sigma)}{d \sigma^2}\vert_{\sigma=1}\leq 0$ we obtain $\frac{d^2 S(\varphi_\omega^\sigma)}{d \sigma^2}< 0\ \forall \sigma >1$ and consequently $\frac{d S(\varphi_\omega^\sigma)}{d \sigma}$ decreasing. Being $\frac{d S(\varphi_\omega^\sigma)}{d \sigma}\vert_{\sigma=1}=0$ we finally obtain that $S_\omega(\varphi^\sigma_\omega)<S_\omega(\varphi_\omega)\ \forall \sigma >1 $ as claimed. Finally, using properties \ref{Q=0} and \ref{Q'} of Proposition \ref{derivateS} we get $ \frac{d}{d\sigma}Q(\varphi_\omega^\sigma)=\frac{d}{d\sigma}S_\omega(\varphi_\omega^{\sigma})+\sigma\frac{d^2}{d\sigma^2}S_\omega(\varphi_\omega^{\sigma})<0$ by using the monotonicity properties just proved. By \ref{Q=0} one finally gets $Q(\varphi_\omega^\sigma)<Q(\varphi_\omega)=0.$ The proof is completed thanks to $\lim_{\sigma\to 1}\|\varphi_\omega^\sigma-\varphi_\omega \|_{\DD^{\frac{1}{2}}}=0$\ .
\end{proof}
\section*{Acknowledgments.}
\noindent The authors are grateful to their friend and colleague Claudio Cacciapuoti for useful discussions. D. Noja acknowledges for funding the EC grant IPaDEGAN (MSCA-RISE-778010). The authors acknowledge the support of the Gruppo Nazionale di Fisica Matematica (GNFM-INdAM).

\end{document}